\documentclass[reqno]{amsart}

\usepackage{amsmath}
\usepackage{amsthm}
\usepackage{amsfonts}
\usepackage{amssymb}
\usepackage{enumitem}
\usepackage{tikz}
\usepackage{tikz-cd}
\usepackage[colorlinks=true,linkcolor=blue,citecolor={blue}]{hyperref}
\usepackage{hyperref}
\usepackage[nocompress]{cite}
\usepackage{mathrsfs}
\usepackage{subcaption}
\usepackage{todonotes}

\DeclareMathOperator{\var}{Var}

\DeclareMathOperator{\He}{He}

\newcommand{\E}{\mathbb{E}}

\renewcommand{\P}{\mathbb{P}}

\newcommand{\eps}{\varepsilon}

\newcommand{\R}{\mathbb{R}}
\newcommand{\N}{\mathbb{N}}
\newcommand{\Z}{\mathbb{Z}}

\def\R{\mathbb{R}}

\newcommand{\dd}{\mathrm{d}}

\theoremstyle{plain}
\newtheorem{theorem}{Theorem}[section]

\newtheorem{lemma}[theorem]{Lemma}
\newtheorem{corollary}[theorem]{Corollary}

\newtheorem{proposition}[theorem]{Proposition}

\theoremstyle{definition}
\newtheorem{definition}[theorem]{Definition}
\newtheorem{example}[theorem]{Example}

\theoremstyle{remark}
\newtheorem{remark}[theorem]{Remark}

\numberwithin{equation}{section}

\begin{document}
	
	\title[Random polynomials and finite free cumulants]{Critical points of random polynomials and finite free cumulants}
	
		\author{Octavio Arizmendi }
		\address{Centro de Investigacion en Matem\'aticas. A.C., Jalisco S/N, Col. Valenciana CP: 36023 Guanajuato, Gto, Mexico }
		\email{octavius@cimat.mx}
		\thanks{O. Arizmendi was supported by the grant Conahcyt A1-S-9764}
	
		\author{Andrew Campbell}
		\address{Institute of Science and Technology Austria, Am Campus 1, 3400 Klosterneuburg, Austria}
		\email{andrew.campbell@ist.ac.at}
		\thanks{A. Campbell is supported by Austrian Science Fund (FWF) under
the Esprit Programme grant ESP4314224.}
		
		\author{Katsunori Fujie}
		\address{Department of Mathematics, Kyoto University,
        Kitashirakawa, Oiwake-cho, Sakyo-ku, Kyoto, 606-8502, Japan}
		\email{email: fujie.katsunori.42m@st.kyoto-u.ac.jp}
		\thanks{K. Fujie was supported by JSPS Open Partnership Joint Research Projects Grant Number JPJSBP120209921 and JSPS Research Fellowship for Young Scientists PD (KAKENHI Grant Number 24KJ1318).}
	
	\begin{abstract}
		A result of Hoskins and Steinerberger [\emph{Int. Math. Res. Not.}, (13):9784–9809, 2022] states that repeatedly differentiating a sequence of random polynomials with independent and identically distributed mean zero and variance one roots will result, after an appropriate rescaling, in a Hermite polynomial. We use the theory of finite free probability to extend this result in two natural directions: (1) We prove central limit theorems for the fluctuations around these deterministic limits for the polynomials and their roots. (2) We consider a generalized version of the Hoskins and Steinerberger result by removing the finite second moment assumption from the roots. In this case the Hermite polynomials are replaced by a random Appell sequence conveniently described through finite free probability and an infinitely divisible distribution.
		We use finite free cumulants to provide compact proofs of our main results with little prerequisite knowledge of free probability required. 
	\end{abstract}

	\maketitle 
	

	\section{Introduction} \label{sec:intro}
    Our goal is to extend existing results on the convergence of roots of derivatives of random polynomials from the finite free probability perspective. We aim to apply the theory of \emph{finite free cumulants}, first developed by \cite{Arizmendi-Perales2018} to linearize the \emph{finite free additive convolution} $\boxplus_{N}$, to study the critical points of random polynomials of the form \begin{equation}\label{eq:random polynomial definition}
		p_N(x)=\prod_{j=1}^{N}\left(x-X_{j}\right),
	\end{equation} where $X_{1},X_{2},\dots$ are independent copies of a random variable $X$ distributed according to some probability measure $\mu$. We are specifically interested in the roots of the polynomial \begin{equation}\label{eq:diff op def}
	\partial_{k|N}p_N:=\frac{(N-k)!}{N!}p^{(N-k)}_{N},
\end{equation} as $N\rightarrow\infty$, where $p^{(N-k)}_{N}$ is the $(N-k)$-th derivative of $p_N$. When $\mu$ is supported on the complex plane and $N-k$ is small, e.g.\ finite or growing very slowly in $N$, there are several papers \cite{Kabluchko-Seidel2019,ORourke-Williams2019,ORourke-Williams2020,Michelen-Vu2024,Michelen-Vu2022,Pemantle-Rivin2013,Hanin2015Gauss,Hanin2017} showing that the roots of $\partial_{k|N}p_N$ are very close to the roots of $p_N$. For measures supported on the real line, significantly larger ranges of $N-k$ can be described \cite{Hoskins-Kabluchko2021,Arizmendi-GarzaVargas-Perales2023,Campbell-ORourke-Renfrew2024even,campbell2025freeinfinitedivisibilityfractional,Gorin-Klepttsyn2020universal,Hoskins-Steinerberger2022,Arizmendi-Fujie-Perales-Ueda2024,Gorin-Marcus2020}. We will consider the situation when $k$ is finite. It can be seen from interlacing that the roots will naturally contract towards their mean under repeated differentiation, so we will often consider the operator $\mathcal{D}_{a}$ on polynomials defined  by \begin{equation}
\left[\mathcal{D}_{a}p\right](x)=a^{\deg(p)}p\left(\frac{x}{a}\right),
\end{equation} which dilates the roots by $a>0$. In the finite $k$ regime, Hoskins and Steinerberger proved the following.  \begin{proposition}[Hoskins and Steinerberger \cite{Hoskins-Steinerberger2022}]\label{thm:HS}
Let $p_N$ be as in \eqref{eq:random polynomial definition} and assume the measure $\mu$ is supported on the real line, has mean $0$, variance $1$, and moments of all orders. If \begin{equation}\label{eq:shifted der}
	\tilde{p}_{k,N}(x)=\mathcal{D}_{\sqrt{N}}\partial_{k|N}p_{N}\left(x+\frac{1}{N}\sum_{j=1}^{N}X_{j}\right),
\end{equation} then, for any fixed $k\in\N$\begin{equation}\label{eq:polynomial limit}
	\lim\limits_{N\rightarrow\infty}\tilde{p}_{k,N}\left(x\right)=\He_{k}(x),
\end{equation} pointwise, almost surely, where $\He_{k}$ is the degree $k$ Hermite polynomial\begin{equation}\label{eq:Hermite def}
\He_k(x)=\sum_{j=0}^{\lfloor \frac{k}{2}\rfloor}\frac{k!(-1)^{j}}{j!(k-2j)!}\frac{x^{k-2j}}{2^{j}}.
\end{equation} 
\end{proposition}  In \cite{Campbell-ORourke-Renfrew2024even},  it was pointed out that the moment assumptions on $\mu$ can be relaxed to just two finite moments.
Our main contribution is to extend this convergence of random polynomials with independent roots in two directions:\begin{enumerate}
\item First, we prove central limit theorems (CLTs) for the polynomial $\tilde{p}_{k,N}$ and its roots around their deterministic limits given in Proposition \ref{thm:HS}. We prove that these fluctuations are Gaussian, and the variance depends only on $\mu$ through its fourth moment. Notably, the variance for the polynomial CLT depends only on $x$ through a lower degree Hermite polynomial. See Theorems \ref{thm:critical point CLT} and \ref{thm:polynomial CLT} for detailed statements. 

\item Second, we prove a generalized version of Proposition \ref{thm:HS} for a sequences of random polynomials $q_{N}$ whose roots may not be in the domain of attraction of the Gaussian distribution, but rather some possibly different infinitely divisible distribution. The Hermite polynomials  in \eqref{eq:polynomial limit} will need to be replaced by a random \emph{Appell sequence} depending only on the infinitely divisible distribution. This result applies to random polynomials of the form \eqref{eq:random polynomial definition} when $\mu$ fails to have two finite moments (giving an analogue for stable distributions) and random polynomials with a very small number of non-zero roots (giving an analogue for the Poisson limit theorem). See Theorem, \ref{thm:ID limiting polynomial} below.
\end{enumerate} In both results, we use the finite free cumulants as an efficient way of tracking how fluctuations of the original iid roots propagate through repeated differentiation.
We limit our study here to iid roots, but our techniques should be applicable to any sequence of random real-rooted polynomials where the statistics of the original roots are well understood, e.g.\ characteristic polynomials of random matrices.

\subsection{Notation and a short introduction to finite free probability} Finite free probability arose from Marcus, Spielman, and Srivastava's celebrated work \cite{Marcus-Spielman-Srivastava2015-1,Marcus-Spielman-Srivastava2015-2,Marcus-Spielman-Srivastava2022} on interlacing families. As our goal is to demonstrate how finite free cumulants provide an efficient and natural way of computing limits of the form \eqref{eq:polynomial limit}, we will only provide basic definitions required for the proofs. We encourage the interested reader to see \cite{Marcus-Spielman-Srivastava2022,Marcus2021,Arizmendi-Perales2018} for a more thorough introduction. 

We will use $\Rightarrow$ to denote convergence in distribution and $\mathcal{N}(a,\sigma^2)$ to denote the Gaussian distribution with mean (mean vector) $a$ and variance (covariance matrix) $\sigma^2$. Let $\mathrm{P}_{\ell}(\R)$ be the set of degree at most $\ell$ polynomials with real coefficients.
We note that pointwise convergence, uniform convergence on compact subsets, and convergence of the coefficients as vectors in $\R^{\ell+1}$ are all equivalent on $\mathrm{P}_{\ell}(\R)$.
Thus, we will collectively refer to all of these as \emph{convergence in $\mathrm{P}_{\ell}(\R)$}, and convergence in distribution of random polynomials with respect to these topologies as \emph{convergence in distribution in $\mathrm{P}_{\ell}(\R)$}.   

For natural numbers $N$ and $k$ we shall use $(N)_k=N(N-1)\cdots(N-k+1)$ to denote the falling factorial. A partition, $\pi = \{ V_1, \ldots, V_r \}$ of $[j] := \{ 1,\ldots, j\}$ is a collection of pairwise disjoint, non-empty, sets $V_i$ such that $\cup_{i=1}^{r} V_i = [j]$. We refer to $V_i$ as a block of $\pi$, and denote the number of blocks of $\pi$ as $|\pi|$. The set of all partitions of $[j]$ is denoted by $\mathcal{P}(j)$.
Additionally, the set of \emph{non-crossing} partitions, see for example \cite{Mingo-Speicher2017} for a definition, will be denoted by $\mathcal{NC}(j)$. We write $1_j = \{\{1,2, \ldots, j \}\}$ to denote the maximal element of $\mathcal{P}(j)$ with respect to reverse refinement.  For a partition $\pi =  \{ V_1, \ldots, V_r \}$ and a sequence of numbers $\{c_n\}_{n=1}^\infty$, we use the notation
\begin{equation} \label{eq:cpinotation}
	c_{\pi}:= \prod_{i=1}^{r}  c_{|V_i|}.
\end{equation} 

For degree $N$ monic polynomials $p(x)=x^{N}+\sum_{k=1}^{N}(-1)^{k}a_{k}x^{N-k}$ and $q(x)=x^{N}+\sum_{k=1}^{N}(-1)^kb_{k}x^{N-k}$ the  \emph{finite free additive convolution} $p\boxplus_{N}q$ of $p$ and $q$ is defined by \begin{equation*}
	p\boxplus_{N}q(x):=x^{N}+\sum_{k=1}^{N}(-1)^{k}\sum_{i+j=k}\frac{(N-i)!(N-j)!}{N!(N-k)!}a_{i}b_{j}.
\end{equation*}
For a degree $N$ polynomial $p$ with roots $\lambda_{1},\dots, \lambda_{N}$, the $j$-th moment of $p$, denoted $m_{j}(p)$, is the $j$-th moment of its empirical root measure
\begin{equation*}
m_{j}(p)=\frac{1}{N}\sum_{k=1}^{N} \lambda_{k}^{j}.
\end{equation*}
For a degree $N$ polynomial $p(x)=x^{N}+\sum_{k=1}^{N}(-1)^{k}a_{k}x^{N-k}$, the finite free cumulants $\kappa_{1}^{N}(p),\dots,\kappa_{N}^{N}(p)$ can be defined implicitly via \begin{equation}\label{eq:finite cumulant def}
a_{k}=\frac{(N)_{k}}{N^{k}k!}\sum_{\pi\in\mathcal{P}(k)}(-1)^{k-|\pi|}N^{|\pi|}\left[\prod_{V\in\pi}(|V|-1)!\right]\kappa_{\pi}^{N}(p).
\end{equation} It is worth noting that $\kappa_{j}^{\ell}(\He_{\ell})=\ell\delta_{j2}$,  for $1\leq j\leq \ell$. Our main results rely primarily on the following results on finite free cumulants \begin{lemma}[Proposition 3.4 in\cite{Arizmendi-Fujie-Perales-Ueda2024}]\label{lem:cumulant-der formula}
Let $p$ be a degree $N$ polynomial. Then, for any $1\leq j\leq k\leq N$ \begin{equation}
	\kappa_{j}^{k}(\partial_{k|N}p)=\left(\frac{k}{N}\right)^{j-1}\kappa_{j}^{N}(p).
\end{equation}
\end{lemma}

\begin{lemma}[See the proof of Theorem 5.4 in\cite{Arizmendi-Perales2018}]\label{lem:finite cumulant leading order}
	Let $p$ be a degree $N$ polynomial. Then, for any $1\leq j\leq N$ \begin{equation}\label{eq:leading order of cumulants}
		\kappa_{j}^{N}(p)=\frac{N^{j}}{(N)_{j}}\left[m_{j}(p)-\sum_{\substack{\sigma\in\mathcal{NC}(j)\\
				\sigma\neq 1_{j}}}\frac{Q_{\sigma}(N)}{N^{j+1-|\sigma|}}\kappa_{\sigma}^{N}(p)\right],
	\end{equation} where each $Q_{\sigma}(N)$ is a monic degree $j+1-|\sigma|$ polynomial  in $N$.
\end{lemma} It will also be convenient to represent polynomials by their \emph{finite} $R$-transform. \begin{definition}\label{def: R def}
     Let $p$ be a degree $N$ polynomial and let $P$ be a formal power series such that $P\left(\frac{\dd}{\dd x}\right)x^{N}=p(x)$. Then, the finite $R$-transform of $p$ is the truncated formal power series \begin{equation}
         R_{p}(s):=-\frac{P'(Ns)}{P(Ns)}\mod s^{N}.
     \end{equation}
\end{definition}

\begin{lemma}[See \cite{Marcus2021,Marcus-Spielman-Srivastava2022,Arizmendi-Perales2018}]\label{lem:R transform and Fourier lemma}
    Let $p$ and $q$ be degree $N$ polynomials, and let $P$ and $Q$ be formal power series such that $P\left(\frac{\dd}{\dd x}\right)x^{N}=p(x)$ and $Q\left(\frac{\dd}{\dd x}\right)x^{N}=q(x)$. Then,     $p\boxplus_{N}q(x)=P\left(\frac{\dd}{\dd x}\right)Q\left(\frac{\dd}{\dd x}\right)x^N$.
     Moreover, the coefficients of $R_{p}$ are the finite free cumulants of $p$, i.e.\ \begin{equation}
        R_{p}(s)=\kappa_{1}^{N}(p)+\kappa_{2}^{N}(p)s+\cdots+\kappa_{N}^{N}(p)s^{N-1}.
    \end{equation}
\end{lemma} 

We end this section by recalling that finite free cumulants are additive, namely, $\kappa_{j}^{N}(p \boxplus_N q) = \kappa_{j}^{N}(p) + \kappa_{j}^{N}(q)$ for $j=1, \dots, N$. It follows that the same is true for the finite R-transform, $R_{p \boxplus_N q} = R_{p} + R_{q}$. 

\subsection*{Acknowledgments}
Part of this  work is based on research conducted during the workshop   ``Recent Perspectives on Non-crossing Partitions through Algebra, Combinatorics, and Probability" hosted by the Erwin Schr\"odinger International Institute for Mathematics and Physics (ESI). We thank the ESI for its support.

We would like to thank Daniel Perales for helpful comments and discussions at early stages of this paper.

\section{Main results}\label{sec:Main results}

\subsection{Central limit theorems of repeated differentiation} For any real rooted degree $\ell$ polynomial $p$, we let $\mathbf{z}(p)$ denote the root vector of $p$ in the Weyl chamber \begin{equation*}
	\mathbb{W}^{\ell}=\left\{(x_{1},\dots,x_{\ell})^{\mathrm{T}}\in\R^{\ell}\ \big|\  x_{1}\geq x_2\geq \cdots\geq x_{\ell} \right\}.
\end{equation*} For any $\ell\in\N$, we consider the positive semi-definite matrix $\Sigma^{(\ell),\mathbf{m}}$ with entries \begin{equation}
	\Sigma^{(\ell),\mathbf{m}}_{ij}=\frac{ij}{4}\left(m_{4}(\mu)-1\right)m_{i}(\He_{\ell})m_{j}(\He_{\ell}).
\end{equation} As we shall see later, this matrix captures the covariance of the moments of $\tilde{p}_{\ell,N}$. To recover the covariance matrix of the roots we will need a linear transformation. Let $V$ be the (scaled) $\ell\times \ell$ Vandermonde matrix for the roots $z_{1,\ell}> z_{2,\ell}> \cdots > z_{\ell,\ell}$ of $\He_{\ell}$ with entries \begin{equation}
	V_{ij}=\frac{j}{\ell}z^{j-1}_{i,\ell},
\end{equation} and let $L=V^{-1}$.

\begin{theorem}\label{thm:critical point CLT}
	Let $\mu$ be a measure on the real line of mean $0$, variance $1$, and having all finite moments. Let $\tilde{p}_{\ell,N}$ be the polynomials defined by \eqref{eq:shifted der}. Then, \begin{equation}
		\sqrt{N}\left[\mathbf{z}\left(\tilde{p}_{\ell,N}\right)-\mathbf{z}\left(\He_{\ell}\right) \right]\Rightarrow \mathcal{N}(0,\Sigma^{(\ell),\mathbf{z}}),
	\end{equation}  where \begin{equation}
		\Sigma^{(\ell),\mathbf{z}}:=L^{\mathrm{T}}\Sigma^{(\ell),\mathbf{m}}L.
	\end{equation}
\end{theorem} \begin{remark}
	The shift in \eqref{eq:shifted der} removed an order $1$ random shift of the roots that is asymptotically converging to a standard Gaussian. Theorem \ref{thm:critical point CLT} captures the order $N^{-1/2}$ fluctuations of the roots. While forming a precise statement of the smaller fluctuations is notationally cumbersome, finite free cumulants can be used to show (at least heuristically) that for large $N$\begin{equation}
		\mathbf{z}\left(\mathcal{D}_{\sqrt{N}}\partial_{\ell|N}p_{N} \right)\approx \mathbf{z}\left(\He_{\ell} \right)+\sum_{k=0}^{\ell-1} N^{-k/2}Z_k\mathbf{v}_{k,\ell},
	\end{equation} where $Z_{1},Z_{2},\dots$ are standard Gaussian random variables and each  $\mathbf{v}_{j,\ell}$ is a deterministic vector depending on $\mu$ only through its first $2j+2$ moments. 
\end{remark}

\begin{figure}[ht]
\centering
\includegraphics[scale = 0.45]{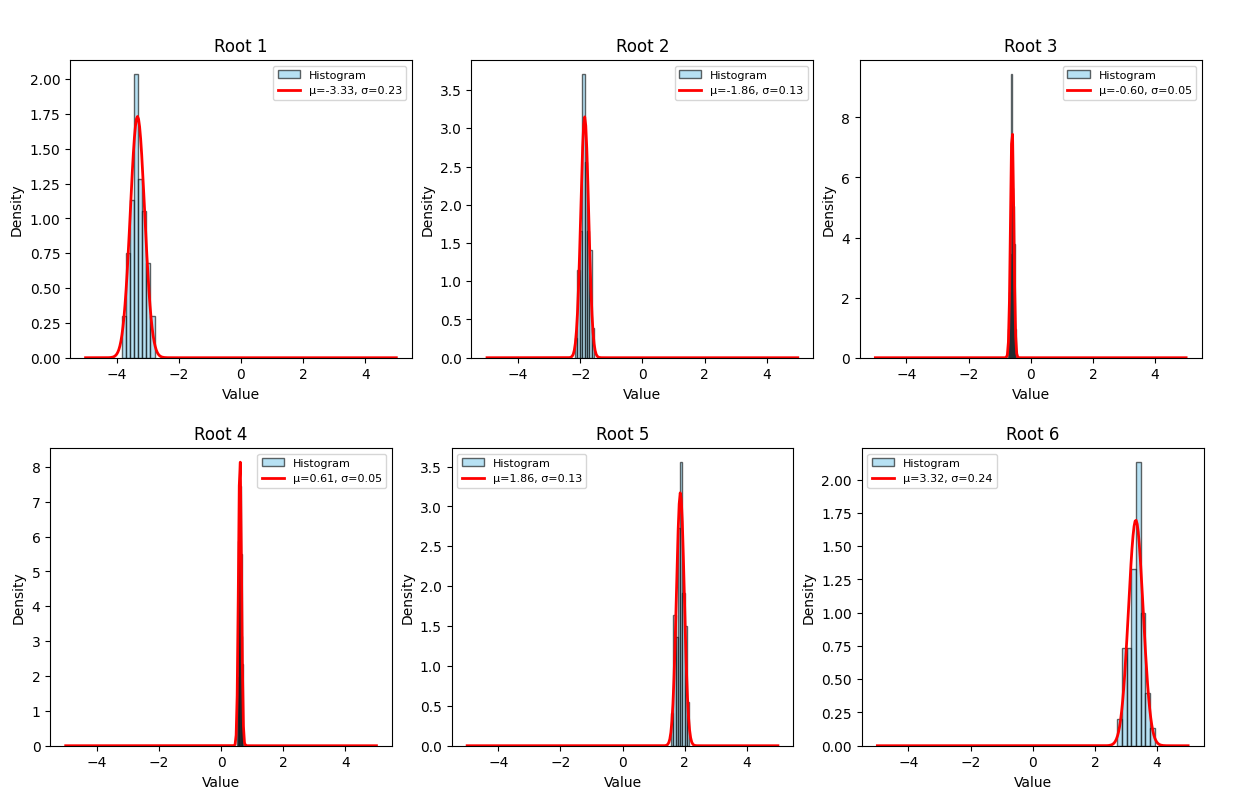}
\caption{Histogram of 100 realizations of the roots of $\tilde{p}_{\ell,N}$ illustrating the fluctuation described in Theorem  \ref{thm:critical point CLT}. For this simulation, $N=100$, $\ell=6$, and we used $\mu\sim\mathcal{N}(0,1)$.  }
\label{Figure1}
\end{figure}

We also have the following CLT for the polynomials. \begin{theorem}\label{thm:polynomial CLT}
	Let $\mu$ be a measure on the real line of mean $0$, variance $1$, and having all finite moments. Let $\tilde{p}_{\ell,N}$ be the polynomials defined by \eqref{eq:shifted der}. Then, \begin{equation}\label{eq:polyCLT}
		\sqrt{N}\left[\tilde{p}_{\ell,N}(x)-\He_{\ell}(x) \right]\Rightarrow \sqrt{m_{4}(\mu)-1}Z\binom{\ell}{2}\He_{\ell-2}(x)
	\end{equation} in $\mathrm{P}_{\ell}(\R)$, where $Z$ is a standard Gaussian random variable. 
\end{theorem} 
Notice that the roots of the right-hand side \label{eq:polyCLT} do not depend on $Z$ or on $m_4$. In Figure \ref{Figure2} we compare the roots of $\sqrt{N}\left[\tilde{p}_{\ell,N}(x)-\He_{\ell}(x)\right]$ with the roots of $\He_{\ell-2}$ for $\ell=8$.


\begin{figure}[ht]
\centering
\includegraphics[width
= 12 cm]{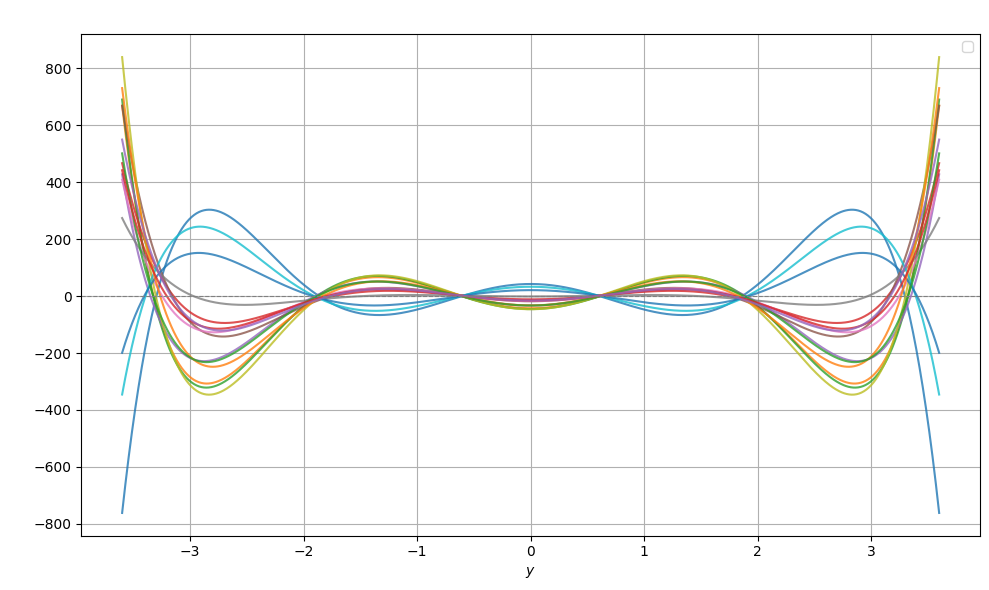}
\caption{Plot of a sample of 15 trials of the random polynomials $\tilde{p}_{8,200}(x)-\He_{8}(x)$. }
\label{Figure2}
\end{figure}

\begin{remark}
	From a probabilistic perspective Theorem \ref{thm:polynomial CLT} is natural, giving the Gaussian fluctuations of a random sequence around a deterministic limit. However, it is not clear to us whether Theorem \ref{thm:polynomial CLT} can be given a natural interpretation from the perspectives of finite free probability since subtracting polynomials is not a nice operation in this theory. We do, however, point out the following. Let $P_{N}$ be a degree $N$ polynomial such that $p_N(x)=P_{N}\left(\frac{\dd}{\dd x} \right)x^{N}$. It is straightforward to show that $\tilde{p}_{\ell,N}(x)=e^{-m_{1}(p_N)\sqrt{N}\frac{\dd}{\dd x}}P_{N}\left(\sqrt{N}\frac{\dd}{\dd x}\right)x^{\ell}$. One interpretation of Theorem \ref{thm:polynomial CLT} is that for any $\ell\in\N$ \begin{equation}
		\lim\limits_{N\rightarrow\infty}\sqrt{N}\left[e^{-m_{1}(p_N)\sqrt{N}s}P_{N}\left(\sqrt{N}s\right)-e^{-s^2/2} \right]= \frac{\sqrt{m_{4}(\mu)-1}}{2}Zs^2e^{-s^2/2}\mod s^{\ell}.
	\end{equation} Thus, \begin{equation}
		e^{-m_{1}(p_N)\sqrt{N}s}P_{N}\left(\sqrt{N}s\right)\approx e^{-s^2/2}\left(1+\frac{\sqrt{m_{4}(\mu)-1}Z}{2\sqrt{N}}s^2\right),
	\end{equation} and hence by Lemma \ref{lem:R transform and Fourier lemma} \begin{equation}
		\tilde{p}_{\ell,N}(x)\approx \He_{\ell}(x)\boxplus_{\ell}x^{\ell-2}\left(x^2+\frac{\sqrt{m_{4}(\mu)-1}}{\sqrt{N}}\binom{\ell}{2}Z\right).
	\end{equation}
	
	We also point out that the right-hand side of \eqref{eq:polyCLT} vanishes at the roots of $\He_{\ell-2}$. Thus, \eqref{eq:polyCLT} and the three-term recurrence of Hermite polynomials imply that the fluctuations of $\tilde{p}_{\ell,N}(x)-\He_{\ell}(x)$ are smaller than $N^{-1/2}$ when \begin{equation}\label{eq:lower order fluctuations}
		x-\frac{1}{G_{\ell}(x)}=0,
	\end{equation} where $G_{\ell}(x)=\frac{1}{\ell}\sum_{j=1}^{\ell}\frac{1}{x-z_{j,\ell}}$ is the Cauchy transform of the uniform distribution on the roots of $\He_{\ell}$. The function $T(x)=x-G_{\ell}(x)^{-1}$, and similar functions, often appear in dynamic descriptions of repeated differentiation and the related process of \emph{fractional free convolution powers} \cite{Shlaykhtenko-Tao2020,Hall-Ho-Jalowy-Kabluchko2023Repeat}. The vanishing of order $N^{-1/2}$ fluctuations at solutions of \eqref{eq:lower order fluctuations} may have a description in the language of PDEs and optimal transport similar to heuristic results derived in \cite{Hall-Ho-Jalowy-Kabluchko2023Repeat}.
\end{remark}

	\subsection{Limiting polynomials for roots outside the Gaussian domain} Theorem \ref{thm:ID limiting polynomial} below provides an analogue of Proposition \ref{thm:HS} for sequences of random polynomials $\{q_N\}_{N=1}^{\infty}$, where for each $N$ the polynomial $q_{N}$ has iid roots, but these roots are not necessarily in the domain of attraction of the Gaussian distribution. We begin with our assumptions. \begin{definition}\label{assump:ID assumptions}
		Let $\{q_N\}_{N=1}^{\infty}$ be a sequence of real rooted polynomials indexed by their degree of the form \begin{equation}
			q_{N}(x)=\prod_{j=1}^{N}\left(x-X_{j,N}\right).
		\end{equation} We say a sequence of polynomials $\{q_{N}\}$ is generated by the L\'evy triple $(c,\sigma^2,\nu)$ if \begin{enumerate}
		\item For any $N\in\N$, $X_{1,N}, X_{2,N},\dots, X_{N,N}$ are iid random variables. 
		\item There exists $\sigma\geq 0$, $c\in\R$, and a measure $\nu$ on $\R$ satisfying \begin{equation}
			\int_{\R}\min(x^2,1)\ \dd \nu(x)<\infty,\quad\text{and}\quad \nu(\{0\})=0,
		\end{equation} such that  $S_{N}=\sum_{j=1}^{N}X_{j,N}$ converges in distribution as $N\rightarrow\infty$ to a random variable $Y$ with log-characteristic function  \begin{equation}\label{eq:ID char}
		\log\E e^{itY}=-\frac{\sigma^2}{2}t^2+itc+\int_{\R}\left(e^{itx}-1-\frac{itx}{1+x^2}\right)\dd\nu(x).
	\end{equation}
	\end{enumerate}
	\end{definition} The limit of repeated differentiation of polynomials generated by a L\'evy triple is captured through a random entire function built from $(c,\sigma^2,\nu)$. We will need a representation of $Y$ in terms of a Poisson point process. We follow the approach of \cite{Ferguson-Klass1972}. We define two functions $N_{+}:(0,\infty)\rightarrow\R$ and $N_{-}:(0,\infty)\rightarrow\R$ as \begin{equation}
	N_{+}(t)=-\nu\left([t,\infty) \right),\text{ and } N_{-}=-\nu\left((-\infty,-t] \right).
\end{equation} These functions are non-decreasing, and we thus define the generalized inverses $N_{\pm}^{-1}(y)=\inf\{t:\ N_{\pm}(t)\geq y\}$ where $y\in(-\infty,0)$. Define the constants $c_{j}$ for $j\in\Z\setminus\{0\}$ by \begin{equation}
c_{j}=\mathrm{sign}(j)\int_{|j|-1}^{|j|}\frac{N_{\mathrm{sign}(j)}^{-1}(-t)}{1+N_{\mathrm{sign}(j)}^{-1}(-t)^2}\dd t.
\end{equation} Let $\{\alpha_{j}\}_{j\in\Z\setminus\{0\}}$ be a Poisson point process with intensity measure $\nu$ indexed such that $\alpha_{1}\geq \alpha_{2}\geq\cdots\geq 0$ and $\alpha_{-1}\leq \alpha_{-2}\leq\cdots\leq 0$. It then follows from \cite[(7)]{Ferguson-Klass1972} that \begin{equation}\label{eq:point process Y def}
Y=c+\sigma Z+\sum_{j\in\Z\setminus\{{0}\}}(\alpha_{j}-c_{j}),
\end{equation}
has log-characteristic function given by \eqref{eq:ID char}, where $Z$ is a standard Gaussian random variable independent of the point process. Define the random entire functions $f$ and $g$ by \begin{equation}\label{eq:random f def}
	f(z):=e^{-Yz-\frac{\sigma^2}{2}z^2}\prod_{j\in\Z\setminus\{0\}}\left(1-\alpha_{j}z\right)e^{\alpha_{j}z}=:e^{-Yz-\frac{\sigma^2}{2}z^2}g(z).
\end{equation}
	
	\begin{theorem}\label{thm:ID limiting polynomial}
		Let $\{q_N\}$ be a sequence of random polynomials generated by the L\'evy triple $(c,\sigma^2,\nu)$ and let $f$ and $g$ be as in \eqref{eq:random f def}. For $\ell\in\N$, let \begin{equation}\label{eq:id normalized derivative}
			\tilde{q}_{\ell,N}(x)=\mathcal{D}_{N}\partial_{\ell|N}q_{N}\left(x\right).
		\end{equation} Then, for any $\ell\in\N$,\begin{equation}\label{eq:id limit}
		\lim\limits_{N\rightarrow\infty}\tilde{q}_{\ell,N}(x)=f\left(\frac{\dd}{\dd x}\right)x^{\ell}=\left(x-Y\right)^{\ell}\boxplus_{\ell}\sigma^{\ell}\He_{\ell}\left(\frac{x}{\sigma}\right)\boxplus_{\ell} g\left(\frac{\dd}{\dd x}\right)x^{\ell},
	\end{equation} in distribution in $\mathrm{P}_{\ell}(\R)$.
	\end{theorem} Not only does finite free probability provide an efficient proof of Theorem \ref{thm:ID limiting polynomial}, it also provides a natural characterization of the limit in terms of the L\'evy triple through \eqref{eq:id limit}.  The Bercovici--Pata bijection \cite{Bercovici-Pata1999} provides an explicit bijection between classical and free infinite divisibility, where the same L\'evy triple determines either the characteristic function or the $R$-transform. As pointed out in \cite{campbell2025freeinfinitedivisibilityfractional}, the free L\'evy-Khintchine representation of the $R$-transform then gives a connection between polynomials of the form $f\left(\frac{\dd}{\dd x}\right)x^{\ell}$ and the free infinitely divisible distribution with $R$-transform $R(s)=-\frac{f'(z)}{f(z)}$. If $f$ is defined by \eqref{eq:random f def} and we view a Poisson point process as a random discrete approximation of its intensity measure $\nu$, then $R(z)=-\frac{f'(z)}{f(z)}$ is a random meromorphic approximation of the $R$-transform determined by the L\'evy triple $(c,\sigma^2,\nu)$ and repeated differentiation provides a way to connect classical and free probability on the level of random polynomial roots.

	\begin{example}
		Some examples of polynomials satisfying the conditions of Theorem \ref{thm:ID limiting polynomial} include:\begin{enumerate}
		\item the polynomials of Proposition \ref{thm:HS}. In this case $f(z)=e^{-Zz-\frac{z^2}{2}}$.	
            
            \item $q_N=\mathcal{D}_{N^{-1/\alpha}}p_{N}$, where $p_N$ are polynomials of the form \eqref{eq:random polynomial definition} for a measure $\mu$ satisfying \begin{equation}\label{eq:domain for stable}
				\begin{aligned}
					\lim\limits_{t\rightarrow\infty}\mu(\{x\in\R:|x|\geq t\} )t^{\alpha}&=c\in(0,\infty)\\
					\lim\limits_{t\rightarrow\infty}\frac{\mu(\{x\in\R:x\geq t\} )}{\mu(\{x\in\R:|x|\geq t\} )}&=\theta\in[0,1],
				\end{aligned}
			\end{equation} for some $\alpha\in (0,2)$. We are not aware of an explicit description of the limiting polynomials beyond \eqref{eq:id limit}. However, the function $f$, which one can think of as the finite free version of the Fourier transform, does have an explicit construction. Let $E_{1},E_{2},\dots$ be i.i.d.\ standard exponential random variables,, $\eps_{1},\eps_{2},\dots$ be i.i.d.\ signed Bernoulli random variables such that $\P(\eps_1=1)=1-\P(\eps_1=-1)=\theta$, and let $\Gamma_{k}=E_1+\cdots+E_{k}$. Then, \begin{equation}
			    f(z)=e^{-Yz}\prod_{k=1}^{\infty}\left(1-\eps_{k}\Gamma_{k}^{-1/\alpha}z \right)e^{\eps_{k}\Gamma_{k}^{-1/\alpha}z}.
		\end{equation} When $\alpha=1$ and $\theta=\frac{1}{2}$, this functions has roots given by a homogeneous Poisson process, and its own behavior under repeated differentiation is the central topic of \cite{Pemantle-Subramanian2017}. Beyond this example, we are not aware of other instances of these functions appearing in the literature. 
            
		\item $q_N(x) = \prod_{j=1}^N(x-X_{j,N})=x^{N-S_N}(x-1)^{S_N}$, where $\{X_{j,N}\}_{j=1}^N$ are i.i.d. Bernoulli random variables with probability $\lambda/N$ and $S_N$ is a Binomial random variable with $\P(S_N=k)=\binom{N}{k}\left(\frac{\lambda}{N}\right)^{k}\left(1-\frac{\lambda}{N}\right)^{N-k}$. In this case $Y$ is a Poisson random variable of mean $\lambda$, $\nu=\lambda\delta_{1}$, $f(z)=(1-z)^{Y}$, and the limiting polynomial in \eqref{eq:id limit} is \begin{equation*}
		   \left(1-\frac{\dd}{\dd x}\right)^{Y}x^{\ell}=(-1)^{Y}Y!x^{\ell-Y}L_{Y}^{(\ell-Y)}(x)= \ell!(-1)^{-\ell}L_{\ell}^{(Y-\ell)}(x),
		\end{equation*} where $L_{\ell}^{(Y-\ell)}$ is the degree $\ell$ Laguerre polynomial of parameter $Y-\ell$. 
		\end{enumerate}
	\end{example}

\section{Proofs of the CLTs} The proofs of the central limit theorems of Section \ref{sec:Main results} follow from Taylor expansions, or the \emph{Delta method} as it is referred to in statistics. We shall use the term Delta method as short hand for the fact that if $\sqrt{N}\left[X_N-\theta\right]\Rightarrow\mathcal{N}(0,\sigma^{2})$ and $g$ is a twice differentiable function such that $g'(\theta)\neq 0$, then \begin{equation}
	\sqrt{N}\left[g\left(X_{N}\right)-g(\theta) \right]\Rightarrow\mathcal{N}(0,\sigma^{2}g'(\theta)^2).
\end{equation} The extension to random vectors is a straightforward generalization using multivariate Taylor expansions. 

We begin with a lemma on the fluctuations of the moments of $p_N$, which follows from straightforward computations. \begin{lemma}\label{lem:Moment CLT for p_N}
	Let $p_N$ be as in \eqref{eq:random polynomial definition}. Then, for any $\ell\in\N$ \begin{equation}
		\sqrt{N}\left[\left(m_{1}(p_N),\dots,m_{\ell}(p_N) \right)^{\mathrm{T}} -\left(m_{1}(\mu),\dots,m_{\ell}(\mu) \right)^{\mathrm{T}}\right]\Rightarrow\mathcal{N}(0,\Sigma), 
	\end{equation} where $\Sigma_{ij}=m_{i+j}(\mu)-m_i(\mu)m_{j}(\mu)$.
\end{lemma} 

Next, we want to compare the finite free cumulants and the free cumulants of $p_N$. While an explicit comparison suitable for our purposes does not seem to have appeared in the literature, the following is essentially contained in the ideas of \cite{Arizmendi-Perales2018,Arizmendi-GarzaVargas-Perales2023}. To simplify notation we write $\kappa_{j}(p_N)=\kappa_{j}(\mu_{p_{N}})$. \begin{lemma}\label{lem:finite to full cumulant comp}
		For any $j\in\N$,\begin{equation}\label{eq:finite to full cumulant comparison}
			\kappa_{j}^{N}(p_N)=\kappa_{j}(p_N)+O\left(\frac{1}{N}\right),
		\end{equation} almost surely. 
	\end{lemma} \begin{proof}[Proof of Lemma \ref{lem:finite to full cumulant comp}]
		We begin by noting that the following proof holds for any sequence of polynomials whose first $j$ moments are uniformly bounded in $N$. Thus, the probabilistic ``almost sure'' conclusion follows from the moment assumptions on $\mu$ and the law of large numbers. The proof will follow by induction on $j$. For $j=1$, $\kappa_{1}^{N}(p_N)=\kappa_{1}({p_N})$ for any polynomial. 
		
		Let $j\geq 2$ and assume \eqref{eq:finite to full cumulant comparison} holds for any $1\leq k\leq j-1$. By the moment-free cumulant formula, see \cite{Mingo-Speicher2017} for example, \begin{equation}
			\kappa_{j}(p_N)=m_{j}(p_N)-\sum_{\substack{\sigma\in\mathcal{NC}(j)\\
					\sigma\neq 1_{j}}}\kappa_{\sigma}(p_N).
		\end{equation}  It then follows from Lemma \ref{lem:finite cumulant leading order} that \begin{equation}
			\begin{aligned}
				\kappa_{j}^{N}(p_N)-\kappa_{j}(p_N)&=\left(\frac{N^{j}}{(N)_{j}}-1 \right)m_{j}(p_N)\\
				&\quad+\sum_{\substack{\sigma\in\mathcal{NC}(j)\\
						\sigma\neq 1_{j}}}\kappa_{\sigma}(p_N)-\frac{N^{j}Q_{\sigma}(N)}{(N)_{j}N^{j+1-|\sigma|}}\kappa_{\sigma}^{N}(p_N)\\
				&=O\left( \frac{1}{N}\right)+\sum_{\substack{\sigma\in\mathcal{NC}(j)\\
						\sigma\neq 1_{j}}}\kappa_{\sigma}(p_N)-\left(1+O\left(\frac{1}{N} \right)\right)\kappa_{\sigma}^{N}(p_N),
			\end{aligned}
		\end{equation} almost surely. Now, for any fixed $\sigma\in\mathcal{NC}(j)$, $\sigma\neq 1_{j}$, we may use the  induction hypothesis to conclude that \begin{equation}
			\kappa_{\sigma}(p_N)-\left(1+O\left(\frac{1}{N} \right)\right)\kappa_{\sigma}^{N}(p_N)=O\left(\frac{1}{N}\right),
		\end{equation} almost surely. Summing over all such $\sigma$ completes the proof. 
	\end{proof}

Theorems \ref{thm:critical point CLT} and Theorem \ref{thm:polynomial CLT} both follow from a central limit theorem, Theorem \ref{thm:finite free cumulant fluc}, for the finite free cumulants. For any degree $\ell$ polynomial $p$, we denote by $\boldsymbol{\kappa}^{\ell}(p)=(\kappa_{1}^{\ell}(p),\dots, \kappa_{\ell}^{\ell} (p))^{\mathrm{T}}$  the vector of its degree $\ell$ finite free cumulants. Additionally, we denote by $\mathbf{m}(p)=(m_{1}(p),\dots,m_{\ell}(p))^{\mathrm{T}}$ the vector of the first $\ell$ moments of $p$. 

\begin{theorem}\label{thm:finite free cumulant fluc}
	Let $\tilde{p}_{\ell,N}$ be as in \eqref{eq:shifted der}. Then, for any $\ell\in\N$ \begin{equation}
		\sqrt{N}\left[\boldsymbol{\kappa}^{\ell}\left(\tilde{p}_{\ell,N}\right)-(0,\ell,0,\dots,0)^{\mathrm{T}}\right]\Rightarrow \mathcal{N}(0,\Sigma^{(\ell)}),
	\end{equation} where the only non-zero entry of $\Sigma^{(\ell)}$ is the $(2,2)$-entry $\ell^2(m_{4}(\mu)-1)$. 
\end{theorem} \begin{corollary}\label{cor:moment fluct}
	Let $\tilde{p}_{\ell,N}$ be as in \eqref{eq:shifted der}. Then, for any $\ell\in\N$ \begin{equation}
		\sqrt{N}\left[\mathbf{m}\left(\tilde{p}_{\ell,N}\right)-\mathbf{m}\left(\He_{\ell}\right)\right]\Rightarrow \mathcal{N}(0,\Sigma^{(\ell),\mathbf{m}}).
	\end{equation}
\end{corollary} 

\begin{proof}[Proof of Theorem \ref{thm:finite free cumulant fluc}]
	 Lemma \ref{lem:Moment CLT for p_N} and application of the Delta method imply that \begin{equation}
		\sqrt{N}\left[\left(\kappa_{1}(p_N),\dots,\kappa_{\ell}(p_N) \right) -\left(\kappa_{1}(\mu),\dots,\kappa_{\ell}(\mu) \right)\right]\Rightarrow\mathcal{N}(0,\Sigma^{\kappa}), 
	\end{equation} for some covariance matrix $\Sigma^{\kappa}$ whose explicit structure is largely unimportant to us.  We do note that\begin{equation}\label{eq:Sigma kappa 22 entry}
		\Sigma^{\kappa}_{22}=m_{4}(\mu)-1,
	\end{equation} which follows from the fact that $\kappa_2(\mu)=\var(\mu)$.  Applying Lemma \ref{lem:finite to full cumulant comp} we see that \begin{equation}\label{eq:original finite free cumulant CLT}
		\sqrt{N}\left[\left(\kappa_{1}^{N}(p_N),\dots,\kappa_{\ell}^{N}(p_N) \right) -\left(\kappa_{1}(\mu),\dots,\kappa_{\ell}(\mu) \right)\right]\Rightarrow\mathcal{N}(0,\Sigma^{\kappa}).
	\end{equation} Note that by Lemma \ref{lem:cumulant-der formula} \begin{equation}\label{eq:der-cumulant relation}
		\kappa_{j}^{\ell}(\tilde{p}_{\ell,N})=\frac{\ell^{j-1}}{N^{\frac{j}{2}-1}}\kappa_{j}^{N}(p_N),
	\end{equation} for any $j\geq 2$, with $\kappa_{1}^{\ell}(\tilde{p}_{\ell,N})=0$. It then follows from \eqref{eq:original finite free cumulant CLT} and \eqref{eq:der-cumulant relation} that \begin{equation}\label{eq:finite cumulant CLT step}
		\sqrt{N}\left[\kappa_{j}^{\ell}(\tilde{p}_{\ell,N})-\ell\delta_{j2}  \right]\Rightarrow\mathcal{N}(0,\ell^2\Sigma^{\kappa}_{22}\delta_{j2}).
	\end{equation} The proof is then completed by combining \eqref{eq:Sigma kappa 22 entry} and \eqref{eq:finite cumulant CLT step}.
\end{proof}

\begin{proof}[Proof of Corollary \ref{cor:moment fluct}]
	Corollary \ref{cor:moment fluct} follows immediately from Theorem \ref{thm:finite free cumulant fluc}, the moment-cumulant formulas \cite[Theorem 4.2]{Arizmendi-Perales2018}, and an application of the Delta method. 
\end{proof}

\begin{proof}[Proof of Theorem \ref{thm:critical point CLT}]
	Let $f:\mathbb{W}^{\ell}\rightarrow\R^{\ell}$ be the map from the ordered roots of a polynomial to the first $\ell$ moments, which is notably an injective smooth function. It is straightforward to see that $Jf(\He_{\ell})=V$ and $Jf(\He_{\ell})^{-1}=L$, where $Jf(\He_{\ell})$ is the Jacobian of $f$ evaluated at $\He_{\ell}$. Then, from Corollary \ref{cor:moment fluct} and the Delta method we have that \begin{equation*}
		\sqrt{N}\left[\mathbf{z}\left(\tilde{p}_{\ell,N}\right)-\mathbf{z}\left(\He_{\ell}\right) \right]\Rightarrow \mathcal{N}(0,\Sigma^{(\ell),\mathbf{z}}),
	\end{equation*} completing the proof. 
\end{proof} 

\begin{proof}[Proof of Theorem \ref{thm:polynomial CLT}]
	From  Proposition \ref{thm:HS}, \eqref{eq:finite cumulant def}, and Theorem \ref{thm:finite free cumulant fluc} \begin{equation}
		\begin{aligned}
			\sqrt{N}\left[\tilde{p}_{\ell,N}(x)-\He_{\ell}(x)\right]&=\sum_{j=2}^{\ell}\sqrt{N}\frac{(\ell)_{j}}{\ell^{j}j!}x^{\ell-j}\\ &\ \ \times\left[\sum_{\pi\in\mathcal{P}(j)}(-1)^{-|\pi|}\ell^{|\pi|}\left(\prod_{V\in\pi}(|V|-1)! \right)\left(\kappa_{\pi}^{\ell}(\tilde{p}_{\ell,N})-\kappa_{\pi}^{\ell}(\He_\ell) \right) \right]\\
			&=\sum_{r=1}^{\lfloor\frac{\ell}{2}\rfloor}\sqrt{N}\left[\frac{(\ell)_{2r}}{\ell^{2r}(2r)!}(2r-1)!!(-1)^{r}\ell^{r}\left(\kappa_{2}^{\ell}(\tilde{p}_{\ell,N})^{r}-\ell^{r} \right) \right]x^{\ell-2r}+o(1)\\
			&\Rightarrow\sqrt{m_{4}(\mu)-1}Z\sum_{r=1}^{\lfloor\frac{\ell}{2}\rfloor}\frac{(-1)^{r}\ell!}{(\ell-2r)!(r-1)!}\frac{x^{\ell-2r}}{2^r}\\
			&=\sqrt{m_{4}(\mu)-1}Z\binom{\ell}{2}\He_{\ell-2}(x).
		\end{aligned}
	\end{equation} This completes the proof.
\end{proof}

\section{Proof of Theorem \ref{thm:ID limiting polynomial}} Define the polynomials $\tilde{q}_{N}(x)=\mathcal{D}_{N}q_{N}\left(x\right)$. We begin with our first lemma on the growth of the cumulants of $\tilde{q}_N$. \begin{lemma}\label{lem:convergence of cumulants to a point process}
	Let $k\in\N$. Then, in distribution  \begin{equation}
			\lim\limits_{N\rightarrow\infty} N^{1-k}\kappa_{k}^{N}(\tilde{q}_N)=\lim\limits_{N\rightarrow\infty}Nm_{k}({q}_N)=\begin{cases}
				Y,\ k=1,\\
				\sigma^2+\sum_{j\in\Z_{\neq0}}\alpha_{j}^2,\ k=2,\\
				\sum_{j\in\Z_{\neq0}}\alpha_{j}^{k},\ k>2,
			\end{cases}
	\end{equation} where  $\{\alpha_{j}\}_{j\in\Z_{\neq0}}$ is a Poisson point process with intensity measure $\nu$ and $Y$ is as defined in \eqref{eq:point process Y def}. Moreover, this convergence is joint for any finite collection of $k_1,k_2,\dots,k_{m}\in\N$.
\end{lemma} \begin{proof} It follows from Definition \ref{assump:ID assumptions} and \cite[Proposition 15.23]{Kallenberg2002}\footnote{Note that the ``only if'' direction of \cite[Proposition 15.23]{Kallenberg2002}, which we use here, does not require symmetry of the variables.} that $\sum_{j=1}^{N}X_{j,N}^{2}$ converges to an infinitely divisible random variable $W$ with log-characteristic function given by \begin{equation}
 	\log\E e^{itW}=\sigma^2it+\int_{(0,\infty)} (e^{itx}-1)\ \dd\nu_{2}(t),
 \end{equation} where $\nu_{2}$ is the push-forward of $\nu$ by $x\mapsto x^{2}$. Consider the random vectors \begin{equation}
     \xi_{k,N}=\begin{pmatrix}
         X_{k,N}\\
         X_{k,N}^{2}
     \end{pmatrix}.
 \end{equation} We can then use\cite[Corollary 15.16]{Kallenberg2002} and the convergence of the marginals of $\xi_{k,N}$ to show by direct computation that \begin{equation}\label{eq:joint convergence}
    \sum_{k=1}^{N}\xi_{k,N}\Rightarrow\begin{pmatrix}
    c+\sigma Z+\sum_{j\in\Z\setminus\{0\}}\alpha_{j}-c_{j} \\
    \sigma^2+\sum_{j\in\Z\setminus\{0\}}\alpha_{j}^{2}
    \end{pmatrix}.
 \end{equation}
 Let $\mathfrak{X}=\overline{\R}\setminus\{0\}$ equipped with the subspace topology.
 Note that from \cite[Corollary 15.16]{Kallenberg2002}, \cite[Theorem 5.3]{Resnick2007}, and \eqref{eq:joint convergence} the random point measures $\tilde\Pi_N=\sum_{j=1}^{N}\delta_{(X_{j,N},X_{j,N}^2)}$ converge to $\tilde\Pi=\sum_{j\in\Z_{\neq0}}\delta_{(\alpha_{j},\alpha_{j}^2)}$ as random Radon measures on $\mathfrak{X}\times\mathfrak{X}$. We may use Skorokhod's representation theorem to assume without loss of generality that this convergence is almost sure. It follows that there exists random permutations $\tau_N$ on $N$ elements such that \begin{equation}\label{eq:pointwise convergence}
 	\left(X_{\tau_N(1),N},X_{\tau_N(2),N},\dots,X_{\tau_N(N),N},0,0,\dots \right)\rightarrow (\alpha_{1},\alpha_{-1},\alpha_{2},\alpha_{-2},\dots)
 \end{equation} in the product topology. We will first prove the result for the moments of $q_{N}$. The cases $k=1$ and $k=2$ are contained in \eqref{eq:joint convergence}.

 We complete the proof for the convergence of the moments by noting that \eqref{eq:pointwise convergence}, the case $k=1$, and the case $k=2$ are the Olshanski--Vershik conditions of \cite{Assiotis-Najnudel2021}, and convergence of the higher moments then follows from \cite[Proposition 2.3]{Assiotis-Najnudel2021} (roughly the proof of their result is that entry-wise convergence and a uniform $\ell^2(\Z)$ bound implies convergence in $\ell^{p}(\Z)$ for $p>2$).

	From the above computation we know that $\max\left(|m_{j}(\tilde{q}_N)|,|\kappa_{j}^{N}(\tilde{q}_{N})|\right)=O(N^{j-1})$ for any $j\in\N$. We then use  \eqref{eq:leading order of cumulants}. For any $\pi\in\mathcal{P}(k)$, $\kappa_{\pi}^{N}(\tilde{q}_N)=O(N^{k-|\pi|})$, and thus to leading order $\kappa_{k}^{N}(\tilde{q}_N)$ is $\frac{N^{k}}{(N)_{k}}m_{k}(\tilde{q}_N)$. The proof then follows from the computation of the moments above. 
\end{proof}

\begin{proof}[Proof of Theorem \ref{thm:ID limiting polynomial}]
	 Let $f$ be the random entire function \eqref{eq:random f def} and let $A_{\ell}(x)=f\left(\frac{\dd}{\dd x}\right)x^{\ell}$ be the polynomials on the right-hand side of \eqref{eq:id limit}. From Definition \ref{def: R def}, $R(s)=-\frac{f'(\ell s)}{f(\ell s)} \mod s^{\ell}$ is the finite $R$-transform of $A_{\ell}$. Moreover, expanding $R(s)$ about $s=0$, we see that \begin{equation}\label{eq:id cumulants}
	\kappa_{k}^{\ell}(A_{\ell})=\begin{cases}
		Y,\ k=1,\\
		\ell\sigma^2+\ell\sum_{j\in\Z_{\neq0}}\alpha_{j}^2,\ k=2,\\
		\ell^{k-1}\sum_{j\in\Z_{\neq0}}\alpha_{j}^{k},\ k>2.
	\end{cases}
\end{equation}
    
    It follows from Lemma \ref{lem:cumulant-der formula} that \begin{equation}
		\kappa_{k}^{\ell}(\tilde{q}_{\ell,N})=\frac{\ell^{k-1}}{N^{k-1}}\kappa_{k}^{N}(\tilde{q}_N).
	\end{equation} Applying Lemma \ref{lem:convergence of cumulants to a point process} we see that \begin{equation}\label{eq:id cumulants limit}
	\boldsymbol{\kappa}^{\ell}(\tilde{q}_{\ell,N})\Rightarrow \boldsymbol{\kappa}^{\ell}(A_{\ell}).
\end{equation} The proof is then complete by noting that mapping coefficients to finite free cumulants is a continuous bijection.
\end{proof}

	\bibliography{RandomCumulants}
\bibliographystyle{abbrv}
	
\end{document}